\documentclass{amsart}
\usepackage[T1]{fontenc}
\usepackage{ae}
\usepackage{aecompl}
\title[Modules for Divergence Zero Vector Fields on a Torus]{Irreducible Modules for the Lie Algebra of Divergence Zero Vector Fields on a Torus}
\date{}
\author{John Talboom}
\newtheorem{thm}{Theorem}[section]
\newtheorem{cor}[thm]{Corollary}
\newtheorem{lem}[thm]{Lemma}
\newtheorem{prop}[thm]{Proposition}

\newtheorem{remark}[thm]{Remark}
\begin{document}
\maketitle
\begin{abstract}
This paper investigates the irreducibility of certain representations for the Lie algebra of divergence zero vector fields on a torus. In \cite{R} Rao constructs modules for the Lie algebra of polynomial vector fields on an $N$-dimensional torus, and determines the conditions for irreducibility. The current paper considers the restriction of these modules to the subalgebra of divergence zero vector fields. It is shown here that Rao's results transfer to similar irreducibility conditions for the Lie algebra of divergence zero vector fields.
\end{abstract}



\section{Introduction}
This paper investigates the irreducibility of tensor modules for the Lie algebra of divergence zero vector fields on a torus. The Lie algebra of derivations of $A=\mathbb{C}[t_1^{\pm1},\dots,t_{N}^{\pm1}]$, denoted by $\mathcal{D}$, may be interpreted as the Lie algebra of polynomial vector fields on an $N$-dimensional torus (see Section 2). In \cite{R} Rao determines the conditions for irreducibility of tensor modules for $\mathcal{D}$. This originates from analogous results established earlier by Rudakov in \cite{Ru} for the Lie algebra of vector fields on an affine space. Tensor modules were also studies by Shen in \cite{S}.

Let $\mathfrak{gl}_N$ be the Lie algebra of $N\times N$ matrices with complex entries, and $\mathfrak{sl}_N$ its subalgebra of trace zero matrices. Then $\mathfrak{gl}_N=\mathfrak{sl}_N\oplus\mathbb{C} I$ where $I$ is the identity matrix. An irreducible $\mathfrak{gl}_N$-module can be constructed by taking $V(\lambda)$,	 the unique finite dimensional irreducible $\mathfrak{sl}_N$-module of highest weight $\lambda$, and specifying that $I$, since it is central, acts by complex number $b$. Denote this $\mathfrak{gl}_N$-module by $V(\lambda,b)$. The tensor product $V(\lambda,b)\otimes A$ becomes a module for a specified action of $\mathcal{D}$.  Rao shows that these modules are irreducible unless $(\lambda,b)=(\omega_k,k)$ for $k\in\{1,\dots,N-1\}$, where $\omega_k$ is a fundamental dominant weight (see Section 3), or if $\lambda$ is zero. In these special cases irreducible quotients and submodules are determined.

The current paper is motivated by the results of \cite{R}, and considers the restriction of these tensor modules to the subalgebra of divergence zero vector fields, denoted $\mathcal{D}_{\text{div}}$. It will be shown in section 2 that the restriction to divergence zero vector fields results in limiting the action to only the $\mathfrak{sl}_N$ part of the modules. Thus the $\mathcal{D}_{\text{div}}$-modules considered here are of the form $V(\lambda)\otimes A$. A streamlined proof of Rao's result was given in \cite{GZ} and the current paper has drawn inspiration from these ideas and applied them here. The result, given in Theorem \ref{results}, mirrors that of Rao's and it is shown that these $\mathcal{D}_{\text{div}}$-modules are irreducible if $\lambda\neq\omega_k,0$.

To understand the cases when $\lambda=\omega_k,0$, the corresponding tensor modules can be realized as the modules of differential $k$-forms on a torus. The tensor modules $V(\omega_k)\otimes A$ for $k\in\{1,\dots,N-1\}$ are isomorphic to $\Omega^k$, the module (over Laurent polynomials) of differential $k$-forms on a torus. The modules of functions $\Omega^0$, and $N$-forms $\Omega^N$, correspond to tensor modules $V(0)\otimes A$, where $V(0)$ is a one-dimensional $\mathfrak{sl}_N$-module. It is useful to consider the de Rham complex of differential forms,
\[\Omega^0\xrightarrow{d}\Omega^1\xrightarrow{d}\dots\xrightarrow{d}\Omega^N\]
(where $d$ is the exterior derivative) since $d$ yields a homomorphism between tensor modules. Thus kernels and images of $d$ are submodules of the $\Omega^k$ and these will be used to determine submodules of the $\mathcal{D}_{\text{div}}$-modules.

In Section 2 the Lie algebra $\mathcal{D}_{\text{div}}$ will be presented, as well as the notation for its homogeneous elements, and a spanning set. The above-mentioned $\mathcal{D}_{\text{div}}$-modules are defined in Section 3 and classified as being minuscule or non minuscule. Two important properties of $\mathcal{D}_{\text{div}}$-modules are given in this section pertaining to the size of these modules and the action of $\mathcal{D}_{\text{div}}$. In Section 4 it is seen that highest weight vectors may always be obtained in modules of non minuscule type. The final section contains the main result, and here it is proven that non minuscule $\mathcal{D}_{\text{div}}$-modules are irreducible, while minuscule modules have irreducible submodules and quotients.
\section{Preliminaries}
Fix $N\in\mathbb{N}$ and consider the column vector space $\mathbb{C}^{N+1}$ with standard basis $\{e_1,\dots,e_{N+1}\}$. Let $(\cdot|\cdot)$ be symmetric bilinear form $(u|v)=u^Tv\in\mathbb{C}$, where $u,v\in\mathbb{C}^{N+1}$ and $u^T$ denotes matrix transpose.

Let $A=\mathbb{C}[t_1^{\pm1},\dots,t_{N+1}^{\pm1}]$ be the Laurent polynomials over $\mathbb{C}$. Elements of $A$ will be presented with the multi-index notation $t^r=t_1^{r_1}\dots t_{N+1}^{r_{N+1}}$ for $r=(r_1,\dots,r_{N+1})\in\mathbb{Z}^{N+1}$. Also, for $i\in\{1,\dots,N+1\}$, let $d_i=t_i\frac{\partial}{\partial t_i}$. The vector space of all of derivations of $A$, $\text{Der}(A)$, forms a Lie algebra called the Witt algebra denoted here by $\mathcal{D}$. Note that $\text{Der}(A)=\text{Span}_{\mathbb{C}}\left\{t^rd_i|i\in\{1,\dots,N+1\}, r\in\mathbb{Z}^{N+1}\right\}$. Homogeneous (in the power of $t$) elements  of $\mathcal{D}$ will be denoted $D(u,r)=\sum_{i=1}^{N+1}u_it^rd_i$ for any $u\in\mathbb{C}^{N+1},r\in\mathbb{Z}^{N+1}$. Its Lie bracket is given by
\[[D(u,r),D(v,s)]=D\left((u|s)v-(v|r)u,r+s\right)\]
for $u,v\in\mathbb{C}^{N+1},r,s\in\mathbb{Z}^{N+1}$.

Geometrically, $\mathcal{D}$ may be interpreted as the Lie algebra of (complex-valued) polynomial vector fields on an $N+1$-dimensional torus, via the mapping $t_j=e^{ix_j}$ for all $j\in\{1,\dots,N+1\}$, where $x_j$ represents the $j$th angular coordinate. This has an interesting subalgebra, the Lie algebra of divergence-zero vector fields, and the corresponding subalgebra of $\mathcal{D}$ is denoted by $\mathcal{D}_{\text{div}}$.
\begin{prop}
Let $u\in\mathbb{C}^{N+1},r\in\mathbb{Z}^{N+1}$. Then $D(u,r)\in\mathcal{D}_{\text{div}}$ if and only if $(u|r)=0$.
\end{prop}
\begin{proof}
Note that under the mapping $t_j=e^{ix_j}$, $\frac{\partial}{\partial x_j}=\frac{\partial t_j}{\partial x_j}\cdot\frac{\partial}{\partial t_j}=it_j\frac{\partial}{\partial t_j}=id_j$, so that $v=\sum_{j=1}^{N+1}f_j(t)d_j\in\mathcal{D}$ can be written $v=-i\sum_{j=1}^{N+1}f_j(t)\frac{\partial}{\partial x_j}$. The divergence of $v$ with respect to the natural volume form in angular coordinates is then $-i\sum_{j=1}^{N+1}\frac{\partial f_j}{\partial x_j}=\sum_{j=1}^{N+1}t_j\frac{\partial f_j}{\partial t_j}$. Thus $\text{div} D(u,r)=\sum_{j=1}^{N+1}u_jr_jt^r=0$ if and only if $(u|r)=0$.
\end{proof}
This proposition uses elements of $\mathcal{D}$ which are homogeneous in $t$, and by linear independence of the powers of $t$ it may be applied to the general case. An element of $\mathcal{D}$ is in $\mathcal{D}_{\text{div}}$ if and only if its homogeneous components are in $\mathcal{D}_{\text{div}}$. It follows that $\mathcal{D}_{\text{div}}=\text{Span}_{\mathbb{C}}\left\{d_a,r_{b}t^rd_a-r_at^rd_{b}|a,b\in\{1,\dots,N+1\}, r\in\mathbb{Z}^{N+1}\right\}$.

Note that for $D(u,r),D(v,s)\in\mathcal{D}_{\text{div}}$, $\left((u|s)v-(v|r)u|r+s\right)=0$, demonstrating that $\mathcal{D}_{\text{div}}$ is closed under the bracket of $\mathcal{D}$. $\mathcal{D}_{\text{div}}$ has Cartan subalgebra $\mathcal{H}=\text{Span}_{\mathbb{C}}\{d_j|j\in\{1,\dots,N+1\}\}=\{D(u,0)|u\in\mathbb{C}^{N+1}\}$.
\section{$\mathcal{D}_{\text{div}}$-Modules}
Let $\mathfrak{sl}_{N+1}$ be the Lie algebra of all $(N+1)\times (N+1)$ traceless matrices with entries in $\mathbb{C}$. Fix a root system $\Phi$, of type $A_N$, with positive roots $\Phi^+$ and simple roots $\Delta=\{\alpha_1,\dots,\alpha_N\}$.

Let $V(\lambda)$ be a finite dimensional irreducible $\mathfrak{sl}_{N+1}$-module of highest weight $\lambda=C_1\omega_1+\dots+C_N\omega_N$, where $\langle\omega_i,\alpha_j\rangle\equiv\frac{2(\omega_i|\alpha_j)}{(\alpha_j|\alpha_j)}=\delta_{ij}$ and $C_i\in\mathbb{Z}_{\geq0}$ for all $i,j\in\{1,\dots,N\}$.
Any weight $\gamma$ of $V(\lambda)$ can be written $\gamma=\lambda-\sum_{i=1}^N\gamma_i\alpha_i$ where the $\gamma_i$'s are non-negative integers. For convenience assign to each weight a \emph{label}, defined as the $N$-tuple $(\langle\gamma,\alpha_1\rangle,\dots,\langle\gamma,\alpha_N\rangle)$. Thus $\lambda$ has label $(C_1,\dots,C_N)$. Note that here a weight of $V(\lambda)$, means only those assigned to a nonzero weight space.

Let $\gamma$ be a weight in $V(\lambda)$ and $\alpha\in\Phi$ (not necessarily simple). Then the set of weights of the form $\gamma+i\alpha$, for $i\in\mathbb{Z}$, is called the \emph{$\alpha$-string through $\gamma$}. let $r,q\in\mathbb{Z}_{\geq0}$ be the largest integers such that $\gamma-r\alpha$ and $\gamma+q\alpha$ are weights of $V(\lambda)$. Then $\gamma+i\alpha$ is a weight for all $-r\leq i\leq q$; i.e. the $\alpha$-string through $\gamma$ is unbroken from $\gamma-r\alpha$ to $\gamma+q\alpha$. Furthermore it can be seen that $\langle\gamma,\alpha\rangle=r-q$. These unbroken sets of weights will be referred to as \emph{weight strings}, and the number of weights in this set its \emph{length}. In this particular case the weight string is in the $\alpha$ direction.

The Weyl group for $A_N$, is isomorphic to $S_{N+1}$, the symmetric group on $N+1$ symbols, and so elements of the Weyl group will be denoted with cycle notation. The reflection through simple root $\alpha_i$ is the transposition $(i,i+1)$, and the reflection through a root of the form $\alpha_i+\dots+\alpha_{i+j}$ is the transposition $(i,i+j+1)$, for $1\leq i<i+j\leq N$.
\begin{lem}\label{size of diagram}
Let $\lambda_{\ell}$ denote the lowest weight of $V(\lambda)$. Write $\lambda_{\ell}=\lambda-\sum_{i=1}^N\kappa_i\alpha_i$, where $\kappa_i\in\mathbb{Z}_{\geq0}$. Then for $1\leq j\leq\lfloor \frac{N+1}{2}\rfloor$,
\[\kappa_j=\kappa_{N+1-j}=\sum_{i=1}^NC_i +\sum_{i=2}^{N-1}C_i +\dots+\sum_{i=j}^{N+1-j}C_i\]
\end{lem}
\begin{proof}\label{}
Consider the Weyl group element
\[w=\prod_{p=0}^{\lfloor\frac{N}{2}\rfloor}(1+p, N+1-p).\]
Then $w(\alpha_i)=-\alpha_{N+1-i}$ for $1\leq i\leq N$ and it follows that $w$ maps the highest weight of the $\mathfrak{sl}_{N+1}$-module to its lowest weight. The above transpositions are disjoint and hence commute. Let 
\[\lambda_j=\left(\prod_{p=0}^{j-1}(1+p,N+1-p)\right)(\lambda)\]
for $j=1,\dots,\lfloor\frac{N}{2}\rfloor+1$ and $\lambda_0=\lambda$. Then
\begin{align*}
\lambda_1&=(1,N+1)(\lambda)\\
&=\lambda_0-\langle\lambda_0,\alpha_1+\dots+\alpha_N\rangle(\alpha_1+\dots+\alpha_N)\\
&=\lambda_0-\left(\sum_{i=1}^NC_i\right)(\alpha_1+\dots+\alpha_N).
\end{align*}
Assume $\lambda_j=\lambda_{j-1}-\sum_{i=j}^{N+1-j}C_i(\alpha_j+\dots+\alpha_{N+1-j})$ for some $j\geq1$. Then
\begin{align*}
\lambda_{j+1}&=(1+j,N+1-j)(\lambda_j)\\
&= \lambda_j-\langle\lambda_j,\alpha_{j+1}+\dots+\alpha_{N-j}\rangle(\alpha_{j+1}+\dots+\alpha_{N-j})\\
&= \lambda_j-\left\langle\lambda-\left(\sum_{i=1}^NC_i\right)(\alpha_1+\dots+\alpha_N)-\left(\sum_{i=2}^{N-1}C_i\right)(\alpha_2+\dots+\alpha_{N-1})\right.\\
&\quad-\dots\\
&\quad- \left.\left(\sum_{i=j}^{N-j}C_i\right)(\alpha_j+\dots+\alpha_{N-j}),\alpha_{j+1}+\dots+\alpha_{N-j}\right\rangle(\alpha_{j+1}+\dots+\alpha_{N-j})\\	
&= \lambda_j-\langle\lambda,\alpha_{j+1}+\dots+\alpha_{N-j}\rangle(\alpha_{j+1}+\dots+\alpha_{N-j})\\
&\quad- \left(\sum_{i=1}^NC_i\right)\left\langle\alpha_1+\dots+\alpha_N,\alpha_{j+1}+\dots+\alpha_{N-j}\right\rangle(\alpha_{j+1}+\dots+\alpha_{N-j})\\
&\quad- \left(\sum_{i=2}^{N-1}C_i\right)\left\langle\alpha_2+\dots+\alpha_{N-1},\alpha_{j+1}+\dots+\alpha_{N-j}\right\rangle(\alpha_{j+1}+\dots+\alpha_{N-j})\\
&\quad-\dots\\
&\quad- \left(\sum_{i=j}^{N-j}C_i\right)\left\langle\alpha_j+\dots+\alpha_{N-j},\alpha_{j+1}+\dots+\alpha_{N-j}\right\rangle(\alpha_{j+1}+\dots+\alpha_{N-j})\\	
&= \lambda_j-\left(\sum_{i=j+1}^{N-j}C_i\right)(\alpha_{j+1}+\dots+\alpha_{N-j}),
\end{align*}
since the angled bracket evaluates to zero in all but the second term. Then by induction
\[w(\lambda)=\lambda_{\lfloor\frac{N}{2}\rfloor+1}=\lambda-\sum_{i=1}^{\lfloor \frac{N}{2}\rfloor+1}\sum_{j=i}^{N+1-i}C_j\sum_{k=i}^{N+1-i}\alpha_k\]
with the final term being zero if $N$ even. Since $w(\lambda)=\lambda_{\ell}$, extracting coefficients of the $\alpha_i$ yields the result.
\end{proof}
Denote $F^{\sigma}(\lambda)= V(\lambda)\otimes\mathbb{C}[q_1^{\pm1},\dots,q_{N+1}^{\pm1}]$ for $\sigma\in\mathbb{C}^{N+1}$. Homogeneous elements of $F^{\sigma}(\lambda)$ will be denoted $v(n)=v\otimes q^n$ for $v\in V(\lambda),n\in\mathbb{Z}^{N+1}$. Then $F^{\sigma}(\lambda)$ becomes a module for $\mathcal{D}_{\text{div}}$ via the action
\[D(u,r).v(n)=(u|n+\sigma)v(n+r)+(ru^T)v(n+r),\]
for $D(u,r)\in\mathcal{D}_{\text{div}}$. Note that $ru^T$ is an $(N+1)\times (N+1)$ matrix with trace zero, as $(u|r)=0$ by assumption, and so the second term involves the $\mathfrak{sl}_{N+1}$ module action. Elements of $\mathfrak{sl}_{N+1}$ are denoted in the usual way, where $E_{ij}$ is the matrix with a 1 in the $(i,j)$ position and zeros elsewhere.

Modules $V(\lambda)$ and $F^{\sigma}(\lambda)$ will be called \emph{minuscule} if  $\lambda=\omega_i$ or $\lambda=0$ for $i\in\{1,\dots,N\}$ (see \cite{BL} 2.11.15). Write $F^{\sigma}(\lambda)=\bigoplus_{n\in\mathbb{Z}^{N+1}}V(\lambda)\otimes q^n$ and let $v(n)\in V(\lambda)\otimes q^n$. Then
$D(e_i,0).v(n)=(n_i+\sigma_i)v(n)$ for all $i\in\{1,\dots,N+1\}$, so that the subspace $V(\lambda)\otimes q^n$ is homogeneous in $q$ and thus $F^{\sigma}(\lambda)$ is a graded module with respect to the action of $\mathcal{H}$. Any submodule $M$ of $F^{\sigma}(\lambda)$ will inherit this gradation and so $M=\bigoplus_{n\in\mathbb{Z}^{N+1}}(V(\lambda)\otimes q^n)\cap M$. It follows that $M$ contains elements of the form $v(n)$ (homogeneous elements).
\begin{prop}\label{actions}
Let $M$ be a submodule of $F^{\sigma}(\lambda)$. If $v(n)\in M$ then $(E_{ij}^kv)(n)\in M$ for any integer $k\geq2$ and any $i\neq j\in\{1,\dots,N+1\}$.
\end{prop}
\begin{proof}\label{}
Note that $D(e_j,re_i)\in\mathcal{D}_{\text{div}}$ for any $r\in\mathbb{Z}$, as $(e_j|e_i)=0$ for $i\neq j$. Fix an integer $k\geq2$, and choose $r_1,\dots,r_k\in\mathbb{Z}\setminus\{0\}$ such that $\sum_{p=1}^kr_p=0$. From the module action given above, $D(e_j,re_i).v(n)=\left((e_j|n+\sigma)+rE_{ij}\right)v(n+re_i)$. Thus
\begin{align*}
D(e_i,r_1e_j)\dots D(e_i,r_ke_j).v(n) &= (K_0+r_1E_{ij})\dots(K_0+r_kE_{ij})v\left(n+\sum_{p=1}^kr_p\right)\\
&= \left(K_0^k+K_1E_{ij}+K_2E_{ij}^2+\dots+K_kE_{ij}^k\right)v(n),
\end{align*}
where $K_0=(e_j|n+\sigma)$, $K_1=K_0^{k-1}\sum_{p=1}^kr_p=0$, and $K_k$ is nonzero. For the case $k=2$, take $r_1=1,r_2=-1$. Then $D(e_j,e_i).D(e_j,-e_i).v(n)=K_0^2v(n)-E_{ij}^2v(n)\in M$, and hence $E_{ij}^2v(n)\in M$. The result follows by induction on $k$.
\end{proof}
\section{Existence of Weight Vectors}	
Recall that the root system of type $\mathfrak{sl}_{N+1}$ has highest root $\theta=\alpha_1+\dots+\alpha_N$. A weight string of maximal length in the $\theta$ direction will be called a \emph{maximal $\theta$-string}.
\begin{prop}\label{max theta string}
	A $\theta$-string through $\lambda$ is a maximal $\theta$-string, and the highest weight of any maximal $\theta$-string has the form $\lambda-\sum_{i=2}^{N-1}\gamma_i\alpha_i$ for some non-negative integers $\gamma_i$.
\end{prop}
\begin{proof}\label{}
Let $\mathcal{S}_{\lambda}$ be the $\theta$-string through $\lambda$. Since $\lambda$ is the highest weight of $\mathcal{S}_{\lambda}$ the lowest weight of $\mathcal{S}_{\lambda}$ is $\lambda-\langle\lambda,\theta\rangle\theta=\lambda-\left(\sum_{i=1}^NC_i\right)\theta$, and thus $\mathcal{S}_{\lambda}$ has length $1+\sum_{i=1}^NC_i$. Let $\mathcal{S}_{\gamma}$ be the $\theta$-string through some other weight $\gamma$, where $\gamma$ is the highest weight of $\mathcal{S}_{\gamma}$ and $\gamma=\lambda-\sum_{j=1}^N\gamma_j\alpha_j$ for some non-negative integers $\gamma_j$.\\

Suppose $\mathcal{S}_{\gamma}$ has length greater than that of $\mathcal{S}_{\lambda}$, say $r+1+\sum_{i=1}^NC_i$ for some positive integer $r$. Then the lowest weight of $\mathcal{S}_{\gamma}$ is $\gamma-\left(r+\sum_{i=1}^NC_i\right)\theta=\lambda-\left(\sum_{j=1}^N\gamma_j\alpha_j\right)-\left(r+\sum_{i=1}^NC_i\right)\theta=\lambda-\sum_{j=1}^N\left(\gamma_j+r+\left(\sum_{i=1}^NC_i\right)\right)\alpha_j$. This contradicts Lemma \ref{size of diagram} since a multiple of $\alpha_1$ (or $\alpha_N$) greater than $\sum_{i=1}^NC_i$ cannot be subtracted from $\lambda$; i.e. no such weight exists. Thus $\mathcal{S}_{\gamma}$ cannot have a length greater than that of $\mathcal{S}_{\lambda}$, so $\mathcal{S}_{\lambda}$ has maximal length.\\

Suppose now that $\mathcal{S}_{\gamma}$ has the same (maximal) length as $\mathcal{S}_{\lambda}$. Then the lowest weight of $\mathcal{S}_{\gamma}$ is $\lambda-\sum_{j=1}^N\left(\gamma_j+\left(\sum_{i=1}^NC_i\right)\right)\alpha_j$.  This contradicts Lemma \ref{size of diagram} in the same way as above unless both $\gamma_1$ and $\gamma_N$ are zero.
\end{proof}
\begin{cor}\label{sl3 theta string}
	Irreducible $\mathfrak{sl}_3$-modules have exactly one maximal $\theta$-string.
\end{cor}
\begin{proof}\label{}
For the $N=2$ case, $\lambda$ is the only weight meeting the criteria of Proposition \ref{max theta string} forcing the only maximal $\theta$-string to be the one through $\lambda$.
\end{proof}
\begin{cor}\label{length at least 3}
	In a non minuscule module the length of any maximal $\theta$-string is at least 3.
\end{cor}
\begin{proof}\label{}
As seen in the proof of Proposition \ref{max theta string}, the length of the $\theta$-string through $\lambda$ is $1+\langle\lambda,\theta\rangle=1+\sum_{i=1}^NC_i$. For a non minuscule module $\sum_{i=1}^NC_i\geq2$.
\end{proof}
\begin{remark}\label{labels}
Let $V(\lambda)$ be a non minuscule module and $N\geq 2$. The label assigned to $\lambda$, $(C_1,\dots,C_N)$, falls into at least one of the following three cases:
\begin{enumerate}
\item[(i)]
	$(C_1,\dots,C_{N-1})$ is a highest weight label for a non minuscule $\mathfrak{sl}_N$-module,
\item[(ii)]	
	$(C_2,\dots,C_N)$ is a highest weight label for a non minuscule $\mathfrak{sl}_N$-module, or
\item[(iii)]	
	$C_1=C_N=1$ and $C_2=\dots=C_{N-1}=0$.
\end{enumerate}
\end{remark}
Applying Lemma \ref{size of diagram} to case (iii) shows that $\lambda_{\ell}=\lambda-2\theta$ and hence it is the lowest weight for the maximal $\theta$-string through $\lambda$. Suppose $\mathcal{S}_{\gamma}$ is a maximal $\theta$-string through the weight $\gamma=\lambda-\sum_{i=1}^n\gamma_i\alpha_i$ for $\gamma_i\in\mathbb{Z}_{\geq0}$. Then $\mathcal{S}_{\gamma}$ has lowest weight $\gamma-2\theta=\lambda-\sum_{i=1}^n\gamma_i\alpha_i-2\theta$. This forces each $\gamma_i$ to be zero, and thus $\mathcal{S}_{\lambda}$ is the only maximal $\theta$-string.

Let $M$ be a nonzero submodule of $F^{\sigma}(\lambda)$. The goal for the remainder of this section is to show that if $F^{\sigma}(\lambda)$ is non minuscule then $v_{\lambda}(n)\in M$ for all $n\in\mathbb{Z}^{N+1}$ where $v_{\lambda}$ is a highest weight vector in $V(\lambda)$. Once $v_{\lambda}(n)$ is obtained it will be used in the next section to generate $V(\lambda)\otimes q^n$ for each $n$, thus generating all of $F^{\sigma}(\lambda)$.
\begin{prop}\label{highest weight component}
$M$ contains a vector $v(n)$, for some $n$, such that $v$ has a nonzero component in $V_{\lambda}$, the highest weight space of $V(\lambda)$.
\end{prop}
\begin{proof}\label{}
Let $v(m)\in M$ for some $m\in\mathbb{Z}^{N+1}$, and $v\in V(\lambda)$. Let $v=\sum_{\gamma\in\Lambda}v_{\gamma}$ be the decomposition of $v$ into weight vectors where $\Lambda$ is the set of weights of $V(\lambda)$. If $v_{\lambda}\neq0$ there is nothing to prove, so suppose $v_{\lambda}=0$. From the definition of the highest weight it follows that there is some $\mu$ in $\Lambda$ such that $E_{ij}v_{\mu}\in V_{\mu'}\neq0$ and $v_{\mu'}=0$ in the above decomposition, for some $i<j\in\{1,\dots,N+1\}$. Applying $D(e_j,e_i)$ to $v(m)$ yields $(e_j|m+\sigma)v(m+e_i)+E_{ij}v(m+e_i)=v'(m+e_i)$, where $v'$ has a nonzero component in $V_{\mu'}$. If $\mu'\neq\lambda$ this process may be repeated by choosing another $(i,j)$-pair in the same way, until the highest weight space has been reached.
\end{proof}
\begin{prop}\label{highest weight vector}
Let $V(\lambda)$ be non minuscule and suppose for some $n\in\mathbb{Z}^{N+1}$ that $v(n)\in M$ has nonzero component in $V_{\lambda}$. Then $v_{\lambda}(n)\in M$, where $v_{\lambda}$ is a highest weight vector of $V(\lambda)$.
\end{prop}
\begin{proof}\label{}
Suppose for some $v(n)\in M$ that $v$ has a nonzero component $v_{\lambda}$ which lies in the highest weight space of $V(\lambda)$. Such a vector exists by Proposition \ref{highest weight component}. Then $v=\sum_{\gamma\in\Lambda}v_{\gamma}$, where $\Lambda$ is a set of weights of $V(\lambda)$ and $v_{\lambda}\neq0$. By Proposition \ref{max theta string} there is a maximal $\theta$-string of length $\langle\lambda,\theta\rangle+1\geq3$ which has highest weight $\lambda$. Note that $E_{N+1,1}$ maps $V_{\mu}$ into $V_{\mu-\theta}$ for any weight $\mu$ of $V(\lambda)$. The goal in what follows is annihilate those components different from $v_{\lambda}$. This will be accomplished by moving $v$ back and forth along the $\theta$-strings.

In the case $N=1$, an irreducible $\mathfrak{sl}_2$-module is a single $\theta$-string. For $N=2$, Corollary \ref{sl3 theta string} says that irreducible $\mathfrak{sl}_3$-modules have only one maximal $\theta$-string. In both cases these strings necessarily have highest weight $v_{\lambda}$. By Proposition \ref{actions}, $E_{N+1,1}^{\langle\lambda,\theta\rangle}v(n)\in M$ where $E_{N+1,1}^{\langle\lambda,\theta\rangle}v\in V_{\lambda-\langle\lambda,\theta\rangle\theta}$. All $v_{\gamma}\neq v_{\lambda}$ are annihilated because their weights either do not lie in the unique maximal $\theta$-string, or if they do, they are of the form $\lambda-i\theta$ for $i\geq1$. In either case $E_{N+1,1}^{\langle\lambda,\theta\rangle}\left(v_{\gamma}\right)=0$, for $\gamma\neq\lambda$, due to maximal length. Applying $E_{1,N+1}^{\langle\lambda,\theta\rangle}$ to $E_{N+1,1}^{\langle\lambda,\theta\rangle}v(n)$ then yields a scalar multiple of $v_{\lambda}(n)$. Now proceed by induction on $N$ with base case $N=2$.

Assume now $N\geq3$ and $v(n)\in M$ where $v=\sum_{\gamma\in\Lambda}v_{\gamma}$. Applying Proposition \ref{actions} yields $E_{N+1,1}^{\langle\lambda,\theta\rangle}v(n)=\sum_{\mu\in\Lambda_*}w_{\mu}(n)\in M$ where each $\mu\in\Lambda_*$ is given by $\mu=\gamma-\langle\lambda,\theta\rangle\theta$ for some $\gamma\in\Lambda$. By maximality of the length of a maximal $\theta$-string, any $\mu\in\Lambda_*$ is a lowest weight in the $\theta$-string which it belongs; i.e. the only $\gamma\in\Lambda$ for which $E_{N+1,1}^{\langle\lambda,\theta\rangle}v_{\gamma}\neq0$ are those which are the highest weight of their maximal $\theta$-string. In particular $\lambda-\langle\lambda,\theta\rangle\theta\in\Lambda_*$. Then $v'(n)=E_{1,N+1}^{\langle\lambda,\theta\rangle}\left(\sum_{\mu\in\Lambda_*}w_{\mu}(n)\right)=\sum_{\gamma\in\Lambda^*}v'_{\gamma}(n)\in M$ where $\lambda\in\Lambda^*$ with $v'_{\lambda}\neq0$ and $\Lambda^*\subset\Lambda$ are weights which occur as the highest weight in their respective maximal $\theta$-string.

Let $\Lambda_1$ be the set of all weights in $V(\lambda)$ of the form $\lambda-\sum_{i=1}^{N-1}a_i\alpha_i$, and $\Lambda_N$ be the set of all weights in $V(\lambda)$ of the form $\lambda-\sum_{i=2}^{N}b_i\alpha_i$,  where $a_i,b_i\in\mathbb{Z}_{\geq 0}$. Let $V(\lambda)_1$ and $V(\lambda)_N$ be the span of all weight vectors with weights in $\Lambda_1$ and $\Lambda_N$ respectively. Then $V(\lambda)_1$ is an $\mathfrak{sl}_N$-module, for the copy of $\mathfrak{sl}_N$ generated by $\{E_{ij}|1\leq i,j\leq N,i\neq j\}\subset\mathfrak{sl}_{N+1}$, with highest weight label $(C_1,\dots,C_{N-1})$. Similarly $V(\lambda)_N$ is an $\mathfrak{sl}_N$-module, for the copy of $\mathfrak{sl}_N$ generated by $\{E_{ij}|2\leq i,j\leq N+1,i\neq j\}\subset\mathfrak{sl}_{N+1}$, and has highest weight label $(C_2,\dots,C_{N})$. Modules $V(\lambda)_1$ and $V(\lambda)_N$ are irreducible since any critical vector they contain is also a critical vector for $\mathfrak{sl}_{N+1}$ in $V(\lambda)$ which is irreducible by assumption.

By Proposition \ref{max theta string} $\Lambda^*$ is a subset of both $\Lambda_1$ and $\Lambda_N$. By remark \ref{labels}, either $V(\Lambda)_1$ or $V(\lambda)_N$ is non minuscule, or there is only one maximal $\theta$-string. In the latter $\Lambda^*$ is empty and so $v'(n)$ is a scalar multiple of $v_{\lambda}(n)$. Otherwise $v'$ lies in a non minuscule $\mathfrak{sl}_N$-module $W$ (either $V(\Lambda)_1$ or $V(\lambda)_N$) and has nonzero component in $V_{\lambda}$. By induction $v_{\lambda}(n)\in W$ and thus $v_{\lambda}(n)\in M$.
\end{proof}
\begin{lem}\label{first component}
If $v_{\lambda}(n)\in M$ then $v_{\lambda}(n+re_1)\in M$ for any $r\in\mathbb{Z}$.
\end{lem}
\begin{proof}
Let $v_{\lambda}(n)\in M$ for some $n\in\mathbb{Z}^{N+1}$. This result is clearly true for $r=0$, so suppose $r\neq 0$. Then
\[D(e_{N+1},re_1).v_{\lambda}(n)=(e_{N+1}|n+\sigma)v_{\lambda}(n+re_1)+rE_{1,N+1}v_{\lambda}(n+re_1),\]
where the second term is zero as $v_{\lambda}$ is a highest weight vector, and if $(e_{N+1}|n+\sigma)\neq0$ then $v_{\lambda}(n+re_1)\in M$. Suppose $(e_{N+1}|n+\sigma)=0$. By proposition \ref{actions} $E_{N+1,1}^2v_{\lambda}(n)\in M$. This is a weight vector whose weight lies on the maximal $\theta$-string through $\lambda$ and is nonzero since the string has length at least 3. So
\begin{multline*}
D(e_{N+1},-re_1).E_{N+1,1}^2v_{\lambda}(n)\\
=(e_{N+1}|n+\sigma)E_{N+1,1}^2v_{\lambda}(n-re_1)-rE_{1,N+1}E_{N+1,1}^2v_{\lambda}(n-re_1).
\end{multline*}
The first term has coefficient zero and the second is a scalar multiple of $E_{N+1,1}v_{\lambda}(n-re_1)$. Then
\begin{multline*}
D(e_{N+1},2re_1).E_{N+1,1}v_{\lambda}(n-re_1)\\
=(e_{N+1}|n-re_1+\sigma)E_{N+1,1}v_{\lambda}(n+re_1)+2rE_{1,N+1}E_{N+1,1}v_{\lambda}(n+re_1)
\end{multline*}
Again the first term is zero and the second term is a nonzero scalar multiple of $v_{\lambda}(n+re_1)$. Thus $v_{\lambda}(n+re_1)\in M$ for all $r\in\mathbb{Z}$.
\end{proof}
\begin{prop}\label{all highest weight vectors}
$v_{\lambda}(m)\in M$ for any $m\in\mathbb{Z}^{N+1}$.
\end{prop}
\begin{proof}
By Propositions \ref{highest weight component} and \ref{highest weight vector} there exists of $v_{\lambda}(n)\in M$ for some $n\in\mathbb{Z}^{N+1}$. Using Lemma \ref{first component} assume that $(e_{1}|n+\sigma)\neq0$. Then for any $r_{N+1}\in\mathbb{Z}$
\begin{multline*}
D(e_1,r_{N+1}e_{N+1}).v_{\lambda}(n)\\
= (e_{1}|n+\sigma)v_{\lambda}(n+r_{N+1}e_{N+1})+r_{N+1}E_{N+1,1}v_{\lambda}(n+r_{N+1}e_{N+1}).
\end{multline*}
Since the first term above is nonzero it follows from Proposition \ref{highest weight vector} that $v_{\lambda}(n+r_{N+1}e_{N+1})\in M$. Thus $v_{\lambda}(n+r_{N+1}e_{N+1})\in M$ for all $r_{N+1}\in\mathbb{Z}$. Let $r=\sum_{i=1}^Nr_ie_i\in\mathbb{Z}^{N+1}$. Then
\begin{multline*}
D(e_{N+1},r).v_{\lambda}(n+r_{N+1}e_{N+1})\\=(e_{N+1}|n+r_{N+1}e_{N+1}+\sigma)v_{\lambda}(n+r+r_{N+1}e_{N+1})
+ \sum_{i=1}^Nr_iE_{i,N+1}v_{\lambda}(n+r+r_{N+1}e_{N+1}),
\end{multline*}
with the second term being zero by the highest weight property of $v_{\lambda}$. When $(e_{N+1}|n+r_{N+1}e_{N+1}+\sigma)\neq0$ it follows by the arbitrary choice of $r$ that $v_{\lambda}(n+r+r_{N+1}e_{N+1})\in M$ for all $r\in\mathbb{Z}^{N}\times\{0\}$.

Suppose there exists $r_{N+1}\in\mathbb{Z}$ such that $(e_{N+1}|n+r_{N+1}e_{N+1}+\sigma)=0$, say when $r_{N+1}=r_{N+1}^*\in\mathbb{Z}$. For any $r=\sum_{i=1}^Nr_ie_i\in\mathbb{Z}^{N+1}$ and $\bar{r}_{N+1}\neq r_{N+1}^*$, the above shows that $v_{\lambda}(n+r+\bar{r}_{N+1})\in M$. By Lemma \ref{first component} $v_{\lambda}(n+ae_1+\sum_{i=2}^Nr_ie_i+\bar{r}_{N+1}e_{N+1})\in M$ where $(e_1|n+ae_1+\sum_{i=2}^Nr_ie_i+\bar{r}_{N+1}e_{N+1}+\sigma)\neq0$. Then
\begin{multline*}
D(e_1,(r_{N+1}^*-\bar{r}_{N+1})e_{N+1}).v_{\lambda}(n+ae_1+\sum_{i=2}^Nr_ie_i+\bar{r}_{N+1}e_{N+1})
\\= (e_1|n+ae_1+\sum_{i=2}^Nr_ie_i+\bar{r}_{N+1}e_{N+1}+\sigma)v_{\lambda}(n+ae_1+\sum_{i=2}^Nr_ie_i+r_{N+1}^*e_{N+1})
\end{multline*}
and so $v_{\lambda}(n+ae_1+\sum_{i=2}^Nr_ie_i+r_{N+1}^*e_{N+1})\in M$. Apply Lemma \ref{first component} again to get that $v_{\lambda}(n+r+r_{N+1}^*e_{N+1})\in M$. This along with the previous paragraph shows that $v_{\lambda}(m)\in M$ for any $m\in\mathbb{Z}^{N+1}$.
\end{proof}
\section{Generating the Module}
\begin{thm}\label{non minuscule lemma}
If $F^{\sigma}(\lambda)$ is non minuscule then it is irreducible as a $\mathcal{D}_{\text{div}}$-module.
\end{thm}
\begin{proof}
From the theory of highest weight modules for semisimple Lie algebras, $V(\lambda)$ is spanned by $v_{\lambda}$ and the vectors $y_1\dots y_kv_{\lambda}$, where, for each $p\in\{1,\dots,k\}$, $y_p=E_{ij}$ for some $(i,j)$-pair with $i>j$. Let $M$ be a nonzero submodule of $F^{\sigma}(\lambda)$. Proceed by induction on $k$ to show that $y_1\dots y_kv_{\lambda}(m)\in M$ for any $m\in\mathbb{Z}^{N+1}$. By Proposition \ref{all highest weight vectors}, $v_{\lambda}(m)\in M$ for any $m\in\mathbb{Z}^{N+1}$, which is the basis for the induction.

Suppose for some $k\geq0$ that $y_1\dots y_kv_{\lambda}(m)\in M$ for any $m\in\mathbb{Z}^{N+1}$ and any $y_1,\dots,y_k$ as above. Then for any $i,j\in\{1,\dots,N+1\}$ with $i>j$
\[D(e_j,e_i).y_1\dots y_kv_{\lambda}(m)
=(e_j|m+\sigma)y_1\dots y_kv_{\lambda}(m+e_i)+E_{ij}y_1\dots y_kv_{\lambda}(m+e_i).\]
Since $y_1\dots y_kv_{\lambda}(m+e_i)\in M$ by the induction hypothesis, $E_{ij}y_1\dots y_kv_{\lambda}(m+e_i)\in M$. Because $m$ is arbitrary, it follows that $E_{ij}y_1\dots y_kv_{\lambda}(n)\in M$ for any $n\in\mathbb{Z}^{N+1}$. Thus $V(\lambda)(m)\subseteq M$ for all $m\in\mathbb{Z}^{N+1}$ and hence $M=F^{\sigma}(\lambda)$.
\end{proof}
The $k$-fold wedge product $\bigwedge^k(\mathbb{C}^{N+1})$ is a highest weight $\mathfrak{sl}_{N+1}$-module via the action
\[X(v_1\wedge\dots\wedge v_k)=\sum_{p=1}^kv_1\wedge\dots\wedge Xv_p\wedge\dots\wedge v_k\]
for any $X\in\mathfrak{sl}_{N+1}  $ and $v_1,\dots,v_k\in\mathbb{C}^{N+1}$. $\bigwedge^k(\mathbb{C}^{N+1})$ has highest weight vector $e_1\wedge\dots\wedge e_k$. It follows from $(E_{ii}-E_{i+1,i+1})(e_1\wedge\dots\wedge e_k)=\delta_{ik}$ that $\bigwedge^k(\mathbb{C}^{N+1})$ has highest weight $\omega_k$ and so $\bigwedge^k(\mathbb{C}^{N+1})\cong V(\omega_k)$ for $1\leq k\leq N$, and $\bigwedge^0(\mathbb{C}^{N+1})\cong\bigwedge^{N+1}(\mathbb{C}^{N+1})\cong V(0)$. For convenience set $\omega_0=\omega_{N+1}=0$ so that $V(0)=V(\omega_0)=V(\omega_{N+1})$.

In \cite{R} some $\mathcal{D}$-submodules of $\bigwedge^k(\mathbb{C}^{N+1})\otimes \mathbb{C}[q_1^{\pm1},\dots,q_{N+1}^{\pm1}]$ are identified. As $\mathcal{D}_{\text{div}}$-modules they are
\[W_k^{\sigma}=\bigoplus_{n\in\mathbb{Z}^{N+1}}\left(\mathbb{C}(n+\sigma)\wedge\mathbb{C}^{N+1}\wedge\dots\wedge\mathbb{C}^{N+1}\right)\otimes q^n,\]
and if $\sigma\in\mathbb{Z}^{N+1}$ there is a larger submodule
\[\tilde{W}_k^{\sigma}=W_k^{\sigma}\oplus\left(\mathbb{C}^{N+1}\wedge\dots\wedge\mathbb{C}^{N+1}\right)\otimes q^{-\sigma}.\]
Let $\tilde{W}_k^{\sigma}=W_k^{\sigma}$ in the case $\sigma\not\in\mathbb{Z}^{N+1}$. For $k=0$ this $\mathcal{D}_{\text{div}}$-module is simply $\mathbb{C}[q_1^{\pm1},\dots,q_{N+1}^{\pm1}]$, and will be denoted $F^{\sigma}(\omega_0)$. The module action reduces to
\[D(u,r).q^n=(u|n+\sigma)q^{n+r}.\]
 If $\sigma\in\mathbb{Z}^{N+1}$ then $F^{\sigma}(\omega_0)$ has submodule $\tilde{W}_0^{\sigma}=\mathbb{C} q^{-\sigma}$.
\begin{lem}\label{wedge weight vector}
Let $N\geq 2$ and $1\leq k\leq N-1$. If $M$ is a submodule of $\bigwedge^k(\mathbb{C}^{N+1})\otimes \mathbb{C}[q_1^{\pm1},\dots,q_{N+1}^{\pm1}]$ which properly contains $\tilde{W}_k^{\sigma}$, then $M$ contains an element of the form $u(n)$, where $u$ is a weight vector	in $\bigwedge^k(\mathbb{C}^{N+1})$, and $n+\sigma\neq 0$.
\end{lem}
\begin{proof}
Since $M$ properly contains $\tilde{W}_k^{\sigma}$, it contains a nonzero vector $v(n)\in F^{\sigma}(\omega_k)\setminus\tilde{W}_k^{\sigma}$ where $n+\sigma\neq0$. Thus $(e_x|n+\sigma)\neq 0$ for some $x\in\{1,\dots,N+1\}$ and so $\{n+\sigma,e_1,\dots,e_{x-1},e_{x+1},\dots,e_{N+1}\}$ is a basis for $\mathbb{C}^{N+1}$. Write 
\[v=\sum_{\substack{p_1,\dots,p_k=1\\p_1,\dots,p_k\neq x}}^{N+1}\gamma_{p_1,\dots,p_k}e_{p_1}\wedge\dots\wedge e_{p_k}\]
(modulo $\tilde{W}_k^{\sigma}$) and assume that the coefficients satisfy $\gamma_{p_1,\dots,p_k}=\text{sgn}(\pi)\gamma_{\pi(p_1,\dots,p_k)}$ for any permutation $\pi$. Since $v\neq0$, at least one of the above coefficients is nonzero, say $\gamma_{i_1,\dots,i_k}$. Assume indices $i_1,\dots,i_k,x,y$ are distinct. Note that in the right hand side of
\begin{multline*}
D(e_x,-e_{y}).D(e_{i_1},e_{y}).v(n)\\=(e_x|n+\sigma)(e_{i_1}|n+\sigma)v(n)
-(e_{i_1}|n+\sigma)E_{yx}v(n)+(e_x|n+\sigma)E_{yi_1}v(n)-E_{yx}E_{yi_1}v(n),
\end{multline*}
the first term is in $M$, while the second and fourth terms are both zero since $e_x$ does not appear in $v$. Thus the third term, call it $v^{(1)}(n)$, is in $M$ and recall $(e_x|n+\sigma)\neq0$. Then,
\begin{align*}
v^{(1)}(n) &= (e_x|n+\sigma)E_{yi_1}v(n)\\
&= (e_x|n+\sigma)\sum_{\substack{p_2,\dots,p_k=1\\p_2,\dots,p_k\neq x}}^{N+1}\gamma_{i_1,p_2,\dots,p_k}e_{y}\wedge e_{p_2}\wedge\dots\wedge e_{p_k}(n)\\
&\quad+ (e_x|n+\sigma)\sum_{\substack{p_1,p_3,\dots,p_k=1\\p_1,p_3,\dots,p_k\neq x}}^{N+1}\gamma_{p_1,i_1,p_3,\dots,p_k}e_{p_1}\wedge e_{y}\wedge e_{p_3}\wedge\dots\wedge e_{p_k}(n)\\
&\quad+ \dots\\
&\quad+ (e_x|n+\sigma)\sum_{\substack{p_1,\dots,p_{k-1}=1\\p_1,\dots,p_{k-1}\neq x}}^{N+1}\gamma_{p_1,\dots,p_{k-1},i_1}e_{p_1}\wedge\dots\wedge e_{p_{k-1}}\wedge e_{y}(n)
\end{align*}
Relabelling indices, the above can be written
\[(e_x|n+\sigma)\sum_{\substack{p_2,\dots,p_k=1\\p_2,\dots,p_k\neq x}}^{N+1}\Gamma e_{y}\wedge e_{p_2}\wedge\dots\wedge e_{p_k}(n),\]
with $\Gamma=\gamma_{i_1,p_2,\dots,p_k}+(-1)\gamma_{p_2,i_1,p_3,\dots,p_k}+\dots+(-1)^{k-1}\gamma_{p_2,\dots,p_{k},i_1}$ which follows from the alternating property of the wedge product. By the assumption on the coefficients, $\gamma_{p_2,\dots,p_s,i_1,p_{s+1},\dots,p_k}=(-1)^{s}\gamma_{i_1,p_2,\dots,p_k}$ for $2\leq s\leq k$, and so
\[v^{(1)}(n)=k(e_x|n+\sigma)\sum_{\substack{p_2,\dots,p_k=1\\p_2,\dots,p_k\neq x}}^{N+1}\gamma_{i_1,p_2,\dots,p_k} e_{y}\wedge e_{p_2}\wedge\dots\wedge e_{p_k}(n).\]
Note that the coefficients in this summation are the same as those from $v$ which have first index equal to $i_1$. Doing this process again, by applying $D(e_x,-e_{i_1})D(e_{i_2},e_{i_1})$ to $v^{(1)}(n)$ in the first step, will yield 
\[v^{(2)}(n)=k^2(e_x|n+\sigma)^2\sum_{\substack{p_{3},\dots,p_k=1\\p_{3},\dots,p_k\neq x}}^{N+1}\gamma_{i_1,i_2,p_{3},\dots,p_k} e_{y}\wedge e_{i_1}\wedge e_{p_3}\wedge\dots\wedge e_{p_k}(n).\]
Repeating (with suitable choices of indices) another $k-2$ times yields
\[k!(e_x|n+\sigma)^k\gamma_{i_1,\dots,i_k}e_y\wedge e_{i_1}\wedge\dots\wedge e_{i_{k-1}}(n)\]
which is a nonzero because both $(e_x|n+\sigma)\neq 0$ and $\gamma_{i_1,\dots,i_k}\neq0$. The wedge product $e_y\wedge e_{i_1}\wedge\dots\wedge e_{i_{k-1}}$ is a weight vector in $\bigwedge^k(\mathbb{C}^{N+1})$ and so the proof is complete.
\end{proof}
\begin{thm}\label{results}
Let $N\geq 1$ and $0\leq k\leq N$.
\begin{enumerate}
\item[(a)]
    If $\lambda\neq0,\omega_k$ then $F^{\sigma}(\lambda)$ is irreducible.
\item[(b)]
    $W_k^{\sigma}$ and $F^{\sigma}(\omega_k)/\tilde{W}_k^{\sigma}$ are irreducible and $F^{\sigma}(\omega_k)/\tilde{W}_k^{\sigma}\cong W_{k+1}^{\sigma}$
\end{enumerate}
\end{thm}
\begin{proof}
Part (a) follows from Theorem \ref{non minuscule lemma}.

For $0\leq k\leq N$, the linear map $\psi_k:F^{\sigma}(\omega_k) \rightarrow F^{\sigma}(\omega_{k+1})$ defined by
\[v_1\wedge\dots\wedge v_k(n) \mapsto (n+\sigma)\wedge v_1\wedge\dots\wedge v_k(n)\]
is a $\mathcal{D}_{\text{div}}$-module homomorphism, where ker$(\psi_k)=\tilde{W}_k^{\sigma}$ and Im$(\psi_k)=W_{k+1}^{\sigma}$. Thus
\[F^{\sigma}(\omega_k)/\tilde{W}_k^{\sigma}\cong W_{k+1}^{\sigma}.\]
Lemma \ref{wedge weight vector} will be used to show that $F^{\sigma}(\omega_k)/\tilde{W}_k^{\sigma}$ is irreducible (and hence that $W_{k+1}^{\sigma}$ is irreducible) for $1\leq k\leq N-1$ with $N\neq2$. The remaining cases are $F^{\sigma}(\omega_0)/\tilde{W}_0^{\sigma}$ and $F^{\sigma}(\omega_N)/\tilde{W}_N^{\sigma}$ for any $N\geq 1$. The $F^{\sigma}(\omega_0)/\tilde{W}_0^{\sigma}$ case is handled first by showing directly that $W_1^{\sigma}$ is irreducible.

Any nonzero submodule $M_1$ of $W_1^{\sigma}=\bigoplus_{n\in\mathbb{Z}^{N+1}}\mathbb{C}(n+\sigma)\otimes q^n$, contains a homogeneous vector $(n+\sigma)\otimes q^n$ for some $n\in\mathbb{Z}^{N+1}$ where $(n+\sigma)\neq0$. Choose $e_x$ so that $(e_x|n+\sigma)\neq0$ and $r\in\mathbb{Z}^{N+1}$ with $(e_x|r)=0$. Then $D(e_x,r).(n+\sigma)\otimes q^n=(e_x|n+\sigma)(n+r+\sigma)\otimes q^{n+r}$ and so $(n+r+\sigma)\otimes q^{n+r}\in M_1$ for any $r\in\{e_x\}^{\perp}\cap\mathbb{Z}^{N+1}$. Let $e_j\neq e_x$, and $r_x\in\mathbb{Z}$. Then
\begin{multline*}
D(e_j-e_x,r_x(e_x+e_j)).(n+r-r_xe_j+\sigma)\otimes q^{n+r-r_xe_j}\\
=(e_j-e_x|n+r-r_xe_j+\sigma)(n+r+r_xe_x+\sigma)\otimes q^{n+r+r_xe_x}.
\end{multline*}
If $(e_j-e_x|n+r-r_xe_j+\sigma)\neq0$ this implies $(m+\sigma)\otimes q^m\in M_1$ for any $m\in\mathbb{Z}^{N+1}$. Otherwise if $(e_j-e_x|n+r-r_xe_j+\sigma)=0$ then consider instead
\begin{multline*}
D(e_j+e_x,r_x(e_x-e_j).(n+r+r_xe_j+\sigma)\otimes q^{n+r+r_xe_j}\\
=(e_j+e_x|n+r+r_xe_j+\sigma)(n+r+r_xe_x+\sigma)\otimes q^{n+r+r_xe_x}
\end{multline*}
so that $(e_j+e_x|n+r+r_xe_j+\sigma)\neq 0$ and again this yields that $(m+\sigma)\otimes q^m\in M_1$ for any $m\in\mathbb{Z}^{N+1}$. Hence $W_1^{\sigma}$ is irreducible.

It remains to show that $F^{\sigma}(\omega_k)/\tilde{W}_k^{\sigma}$ is irreducible for $1\leq k\leq N$. Let $M$ be a submodule of $F^{\sigma}(\omega_k)$ which properly contains $\tilde{W}_k^{\sigma}$.

For $N\geq 2$ with $1\leq k\leq N-1$, Lemma \ref{wedge weight vector} shows that $M$ contains a vector of the form $e_y\wedge e_{i_1}\wedge\dots\wedge e_{i_{k-1}}(n)$, with $n$ such that $(e_x|n+\sigma)\neq0$ for some $x\in\{1,\dots,N+1\}$. For any $r\in\{e_x\}^{\perp}$,
\[D(e_x,r).e_y\wedge e_{i_1}\wedge\dots\wedge e_{i_{k-1}}(n)=(e_x|n+\sigma)e_y\wedge e_{i_1}\wedge\dots\wedge e_{i_{k-1}}(n+r).\]
Thus $e_y\wedge e_{i_1}\wedge\dots\wedge e_{i_{k-1}}(n+r)\in M$ for any $r\in\{e_x\}^{\perp}$. So for $a\not\in\{x,y,i_1,\dots,i_{k-1}\}$,
\begin{multline*}
D(e_x-e_a,r_x(e_x+e_a)).e_y\wedge e_{i_1}\wedge\dots\wedge e_{i_{k-1}}(n+r-r_xe_a)\\
=(e_x-e_a|n+r-r_xe_a+\sigma)e_y\wedge e_{i_1}\wedge\dots\wedge e_{i_{k-1}}(n+r+r_xe_x)
\end{multline*}
If $(e_x-e_a|n+r-r_xe_a+\sigma)\neq0$ this yields that $e_y\wedge e_{i_1}\wedge\dots\wedge e_{i_{k-1}}(m)\in M$ for all $m\in\mathbb{Z}^{N+1}$. If $(e_x-e_a|n+r-r_xe_a+\sigma)=0$ apply instead $D(e_x+e_a,r_x(e_x-e_a))$ to $e_y\wedge e_{i_1}\wedge\dots\wedge e_{i_{k-1}}(n+r+r_xe_a)$ and obtain the same result.

Using the fact that $e_y\wedge e_{i_1}\wedge\dots\wedge e_{i_{k-1}}(m)\in M$ for all $m\in\mathbb{Z}^{N+1}$, any vector of the form $e_{b_1}\wedge e_{b_2}\wedge\dots\wedge e_{b_k}(m)$ can be obtained by doing the following. To replace $e_j\in\{e_y,e_{i_1},\dots,e_{i_{k-1}}\}$ with $e_a$, $a\not\in\{y,i_1,\dots,i_{k-1}\}$ in $e_y\wedge e_{i_1}\wedge\dots\wedge e_{i_{k-1}}(m)$, note that
\begin{multline*}
D(e_j,e_a).e_y\wedge e_{i_1}\wedge\dots\wedge e_j\wedge\dots\wedge e_{i_{k-1}}(m-e_a)\\
=(e_j|m-e_a+\sigma)e_y\wedge e_{i_1}\wedge\dots\wedge e_j\wedge\dots\wedge e_{i_{k-1}}(m)+e_y\wedge e_{i_1}\wedge\dots\wedge e_a\wedge\dots\wedge e_{i_{k-1}}(m).\end{multline*}
The first term on the right hand side is already in $M$ which implies $e_y\wedge e_{i_1}\wedge\dots\wedge e_a\wedge\dots\wedge e_{i_{k-1}}(m)\in M$. This process of swapping out $e_j$'s can be repeated until the desired vector is obtained. Vectors of the form $e_{b_1}\wedge e_{b_2}\wedge\dots\wedge e_{b_k}(m)$ form a basis for $V(\omega_k)(m)$, and thus $M$ must be all of $F^{\sigma}(\omega_k)$. It follows that $F^{\sigma}(\omega_k)/\tilde{W}_k^{\sigma}$ is irreducible.

Consider the case $k=N$ for $N\geq 1$. $M$ contains a nonzero vector $v(n)$, where $v(n)\not\in\tilde{W}_N^{\sigma}$ such that $n+\sigma\neq0$. Suppose $(e_x|n+\sigma)\neq0$ for some $x\in\{1,\dots,N+1\}$, and so $\{n+\sigma,e_1,\dots,e_{x-1},e_{x+1},\dots,e_{N+1}\}$ is a basis for $\mathbb{C}^{N+1}$. From the definition of $\tilde{W}_N^{\sigma}$ it follows that $v(n)$ is a nonzero scalar multiple of $e_1\wedge\dots\wedge e_{x-1}\wedge e_{x+1}\wedge\dots\wedge e_{N+1}(n)$ plus some vector in $\tilde{W}_N^{\sigma}$. Assume without loss of generality that $v(n)=e_1\wedge\dots\wedge e_{x-1}\wedge e_{x+1}\wedge\dots\wedge e_{N+1}(n)$. Note that
\[V(\omega_N)(n)= \mathbb{C} v(n)\oplus(n+\sigma)\wedge\left(\bigwedge^{N-1}\mathbb{C}^{N+1}\right)(n)\]
and so $V(\omega_N)(n)\subset M$. For any $r\in\{e_x\}^{\perp}\cap\mathbb{Z}^{N+1}$, $D(e_x,r).v(n)=(e_x|n+\sigma)v(n+r)$, thus $v(n+r)\in M$. Since $(e_x|n+r+\sigma)=(e_x|n+\sigma)\neq0$ it follows that
\[V(\omega_N)(n+r)= \mathbb{C} v(n+r)\oplus(n+r+\sigma)\wedge\left(\bigwedge^{N-1}\mathbb{C}^{N+1}\right)(n+r)\]
and so $V(\omega_N)(n+r)\subset M$ for any $r\in\{e_x\}^{\perp}\cap\mathbb{Z}^{N+1}$. It remains to show that $V(\omega_N)(n+r+r_xe_x)\subset M$ for any $r_x\in\mathbb{Z}$.

Let $j>x$, $r_x\in\mathbb{Z}\setminus\{0\}$, and $r\in\{e_x\}^{\perp}\cap\mathbb{Z}^{N+1}$. Then
\[D(e_j,r_xe_x).v_{\omega_N}(n+r)=(e_j|n+r+\sigma)v_{\omega_N}(n+r+r_xe_x),\]
where $v_{\omega_N}$ is a highest weight vector in $V(\omega_N)$. If $(e_j|n+r+\sigma)\neq0$ then $v_{\omega_N}(n+r+r_xe_x)\in M$, and by the proof of Lemma \ref{non minuscule lemma} $V(\omega_N)(n+r+r_xe_x)\subset M$. Let $v_{\ell}$ be a lowest weight vector of $V(\omega_N)$, and $j<x$, then
\[D(e_j,r_xe_x).v_{\ell}(n+r)=(e_i|n+r+\sigma)v_{\ell}(n+r+r_xe_x).\]
If $(e_i|n+r+\sigma)\neq0$ then $v_{\ell}(n+r+r_xe_x)\in M$. Again the proof of Lemma \ref{non minuscule lemma} shows how to generate all of $V(\omega_N)(n+r+r_xe_x)$, only this time the $y_p=E_{ij}$ where $i<j$ are applied to lowest weight vector $v_{\ell}(n+r+r_xe_x)$. Suppose for some $r\in\{e_x\}^{\perp}\cap\mathbb{Z}^{N+1}$, say $r=r_0$, that $(e_j|n+r_0+\sigma)=0$ for all $j\neq x$; i.e. $(n+r_0+\sigma)=Ke_x\neq0$. Then for $j\neq x$
\begin{multline*}
D(e_j-e_x,r_x(e_j+e_x)).v(n+r_0-r_xe_j)=K v(n+r_0+r_xe_x)\\
+r_xe_1\wedge\dots\wedge e_{j-1}\wedge e_x\wedge e_{j+1}\wedge\dots\wedge e_{x-1}\wedge e_{x+1}\wedge\dots\wedge e_{N+1}(n+r_0+r_xe_x).
\end{multline*}
Since $(n+r_0+r_xe_x+\sigma)$ is a scalar multiple of $e_x$, the second term above is in $\tilde{W}_N^{\sigma}$, and thus $v(n+r_0+r_xe_x)\in M$. If $(e_x|n+r_0+r_xe_x+\sigma)\neq0$ then
\begin{multline*}
V(\omega_N)(n+r_0+r_xe_x)\\
=\mathbb{C} v(n+r_0+r_xe_x)\oplus(n+r_0+r_xe_x+\sigma)\wedge\left(\bigwedge^{N-1}\mathbb{C}^{N+1}\right)(n+r_0+r_xe_x)
\end{multline*}
and so $V(\omega_N)(n+r_0+r_xe_x)\subset M$. If $(e_x|n+r_0+r_xe_x+\sigma)=0$, then $(n+r_0+r_xe_x+\sigma)=0$, in which case $V(\omega_N)(n+r_0+r_xe_x)\subset\tilde{W}_N^{\sigma}\subset M$. Thus $V(\omega_N)(m)\subset M$ for all $m\in\mathbb{Z}^{N+1}$ and so $M=F^{\sigma}(\omega_N)$.
\end{proof}
\section{Acknowledgements}
Many thanks to Professor Yuly Billig for his expert guidance on this project.



\begin{thebibliography}{5}
\bibitem{BL} S. Billey, V. Lakshmibai, Singular Loci of Schubert Varieties, Birkh{\"a}user, Boston, 2000.
\bibitem{R} S. Eswara Rao, Irreducible Representations of the Lie-Algebra of the Diffeomorphisms of a $d$-Dimensional Torus, J. Algebra 182 (1996), no. 2, 401421. MR1391590 (97f:17007)
\bibitem{GZ} X. Guo and K. Zhao, Irreducible Weight Modules over Witt Algebras, Proceedings of the American Mathematical Society (2011), Vol. 139, No. 7, 2367-2373. S0002-9939(2010)10679-2.
\bibitem{Ru} A. N. Rudakov, Irreducible representations of infinite-dimensional Lie algebras of
Cartan type, Math. USSR Izv. 8 (1974), 836-866.
\bibitem{S} G. Shen, Graded modules of graded Lie algebras of Cartan Type. I. Mixed products of modules, Sci. Sinica Ser. A 29 (1986), No. 6, 570-581
\end{thebibliography}
\end{document}